\documentclass[final,3p]{elsarticle}
\usepackage{graphicx}
\usepackage{amssymb}
\usepackage{amsthm}
\usepackage{amsmath}
\usepackage{booktabs}
\usepackage{mathdots}
\usepackage[algoruled,lined,algosection,longend]{algorithm2e}
\newcommand{\step}{\textit{step}}
\newcommand{\imax}{\textit{imax}}
\newcommand{\tol}{\textit{tol}}
\journal{Applied Mathematics and Computation}
\newcommand{\diag}{\operatorname{diag}}

\newcommand{\inertia}{\operatorname{inertia}}
\newcommand{\assgn}{\ensuremath\mathrel{\mathop:}=}
\newcommand\hpm{\hphantom{-}}

\newtheorem{theorem}{Theorem}[section]

\newtheorem{corollary}[theorem]{Corollary}

\newdefinition{definition}[theorem]{Definition}
\numberwithin{equation}{section}
\numberwithin{figure}{section}
\numberwithin{table}{section}
\begin{document}
\begin{frontmatter}
\title{The antitriangular factorization of skew-symmetric matrices\tnoteref{label1}}
\tnotetext[label1]{This work has been fully supported by Croatian Science Foundation under the project IP-2014-09-3670.}
\author[label2]{Sanja Singer}
\ead{ssinger@fsb.hr}
\address[label2]{University of Zagreb, Faculty of Mechanical Engineering and Naval Architecture, Ivana Lu\v{c}i\'{c}a 5, 10000 Zagreb, Croatia}
\begin{abstract}
  In this paper we develop algorithms for orthogonal similarity
  transformations of skew-symmetric matrices to simpler forms.  The
  first algorithm is similar to the algorithm for the block
  antitriangular factorization of symmetric matrices, but in the case
  of skew-symmetric matrices, an antitriangular form is always
  obtained.  Moreover, a simple two-sided permutation of the
  antitriangular form transforms the matrix into a multi-arrowhead
  matrix.  In addition, we show that the block antitriangular form of
  the skew-Hermitian matrices has the same structure as the block
  antitriangular form of the symmetric matrices.
\end{abstract}
\end{frontmatter}
%
%%%%%%%%%%%%%%%%%%%%%%%%%%%%%%%%%%%%%%%%%%%%%%%%%%%%%%%%%%%%%%%%%%%%%%%%%%%%%%
%
\section{Introduction}
\label{sec:1}
%
%%%%%%%%%%%%%%%%%%%%%%%%%%%%%%%%%%%%%%%%%%%%%%%%%%%%%%%%%%%%%%%%%%%%%%%%%%%%%%
%
Skew-symmetric matrices are significantly less used than symmetric
ones.  Many algorithms designed for symmetric matrices have been
transformed in the course of last two decades to work with the
skew-symmetric and other structured matrices, to avoid the algorithms
for the general, nonstructured, matrices.

Mastronardi and Van Dooren in~\cite{Mastronardi-VanDooren-13} showed
that every symmetric and indefinite matrix $A \in \mathbb{R}^{n \times n}$
can be transformed into a block antitriangular form by orthogonal
similarities.  More precisely, if $\inertia(A) = (n_{-}, n_0, n_{+})$,
$n_1 = \min(n_{-}, n_{+})$, $n_2 = \max(n_{-}, n_{+}) - n_1$,
there exists an orthogonal matrix $Q \in \mathbb{R}^{n \times n}$ such
that
\begin{equation}
  M = Q^T A Q = \begin{bmatrix}
    0 & 0 & 0 & 0 \\
    0 & 0 & 0 & Y^T \\
    0 & 0 & X & Z^T \\
    0 & Y & Z & W
  \end{bmatrix},
  \label{1.1}
\end{equation}
where $Y \in \mathbb{R}^{n_1 \times n_1}$ is nonsingular and lower
antitriangular, $W \in \mathbb{R}^{n_1 \times n_1}$ is symmetric,
$X \in \mathbb{R}^{n_2 \times n_2}$ is symmetric and definite, and
$Z \in \mathbb{R}^{n_1 \times n_2}$.

Bujanovi\'{c} and Kressner in~\cite{Bujanovic-Kressner-16} derived a
computationally effective block algorithm that computes the block
antitriangular factorization (\ref{1.1}).  Unfortunately that
algorithm sometimes fails to detect the inertia.  A new algorithm for
the antitriangular factorization was presented
in~\cite{Laudadio-Mastronardi-VanDooren-16}.

Pestana and Wathen in~\cite{Pestana-Wathen-14} simplified the
algorithm for the special saddle point matrices
\begin{displaymath}
  A = \begin{bmatrix}
    H & B^T \\
    B & 0
  \end{bmatrix},
\end{displaymath}
where $H \in \mathbb{R}^{k \times k}$ is symmetric, but not
necessarily positive definite, and $B \in \mathbb{R}^{m \times k}$,
$m \geq k$. 

In this paper we show that skew-symmetric matrices have antitriangular
form, while skew-Hermitian ones have a block antitriangular form
similar to the block antitriangular form of real symmetric matrices.

In the next section of the paper we constructively prove that every
skew-symmetric matrix can be transformed into lower antitriangular
form, and establish the connection between the number of nontrivial
antidiagonals and the rank of the skew-symmetric matrix.  In
Section~\ref{sec:3} a stable numerical procedure for computing the
antitriangular form is derived.  In Section~\ref{sec:4} we show that
the antitriangular form can be reorganized to the multi-arrowhead
form.  Section~\ref{sec:5} contains the results about block
antitriangular form of Hermitian, and, therefore, skew-Hermitian
matrices.
%
%%%%%%%%%%%%%%%%%%%%%%%%%%%%%%%%%%%%%%%%%%%%%%%%%%%%%%%%%%%%%%%%%%%%%%%%%%%%%%
%
\section{Factorization of a skew-symmetric matrix into antitriangular form}
\label{sec:2}
%
%%%%%%%%%%%%%%%%%%%%%%%%%%%%%%%%%%%%%%%%%%%%%%%%%%%%%%%%%%%%%%%%%%%%%%%%%%%%%%
%
In this section we constructively prove that every skew-symmetric
matrix can be reduced to antitriangular form by orthogonal similarity
transformations.

To this end we use Givens rotations, since Jacobi rotations
$Q_{ij} \assgn Q(i, j, \varphi_{ij})$ cannot annihilate the element at
the position $(i, j)$ in a skew symmetric matrix $A$.  Suppose that
$A_{ij}$ is a skew-symmetric matrix of order $2$, and $Q_{ij}$ is a
rotation.  Then we have
\begin{displaymath}
  Q_{ij}^T A_{ij}^{} Q_{ij}^{} = \begin{bmatrix}
    \hpm\cos \varphi & \sin \varphi \\
    -\sin \varphi    & \cos \varphi
  \end{bmatrix}
  \begin{bmatrix}
    \hpm 0 & a_{ij} \\
    -a_{ij} & 0
  \end{bmatrix}
  \begin{bmatrix}
    \cos \varphi & -\sin \varphi \\
    \sin \varphi & \hpm \cos \varphi
  \end{bmatrix}
  = \begin{bmatrix}
    \hpm 0 & a_{ij} \\
    -a_{ij} & 0
  \end{bmatrix}
  = A_{ij}.
\end{displaymath}
Therefore, we use the Givens rotation $Q_{ij}$ to annihilate the
elements at positions $(i, k)$ and $(k, i)$, $k \neq j$, or at
positions $(k, j)$ and $(j, k)$, $k \neq i$.

\begin{theorem}
  Let $A \in \mathbb{R}^{n \times n}$ be a skew-symmetric matrix.
  Then $A$ can be factored as
  \begin{displaymath}
    A = Q M Q^T,
  \end{displaymath}
  where $Q$ is an orthogonal matrix, and $M$ is an antitriangular matrix.
  \label{thm2.1}
\end{theorem}

\begin{proof}
  The proof is by induction over the number of already annihilated
  antidiagonals of a skew-symmetric matrix $A$.

  Note that $A$ has a zero on its position $(1, 1)$, and this fact
  serves as the basis of induction.
  
  Suppose that after $k-1$ annihilated antidiagonals $M_{k-1}$ has the
  following form,
  \begin{equation}
    M_{k-1}^{} \assgn Q_{k-1}^T A Q_{k-1}^{}
    = \begin{bmatrix}
      \hphantom{-}M_{11}^{} & M_{12}^{} \\
      -M_{12}^T             & M_{22}^{}
    \end{bmatrix},
    \label{2.1}
  \end{equation}
  where
  \begin{equation}
    M_{11}^{} = \begin{bmatrix}
      0       & \cdots   & \cdots  & 0 \\
      \vdots  &          & \iddots & m_{2,k-1} \\
      \vdots  & \iddots  & \iddots & \vdots \\
      0       & -m_{2,k-1} & \cdots & 0
    \end{bmatrix},
    \label{2.2}
  \end{equation}
  while the matrices $M_{12}$ and $M_{22}$ are generally full.  In the
  matrix $Q_{k-1}$ we keep the product of the applied rotations.  If
  $n = k$, we have completed the job.  Otherwise, in the next step we
  annihilate the $k$th antidiagonal.

  First we annihilate elements at positions $(1, k)$ and $(k, 1)$ by
  a rotation $Q_{k,k+1}$ in the plane $(k, k+1)$ that is equal to the
  identity matrix except at the crossings of the $k$th and the
  $(k+1)$th rows and columns, where
  \begin{equation}
    \widehat{Q}_{k,k+1} = \begin{bmatrix}
      \cos \varphi_{k,k+1} & -\sin \varphi_{k,k+1} \\
      \sin \varphi_{k,k+1} & \hpm \cos \varphi_{k,k+1}
    \end{bmatrix}.
    \label{2.3}
  \end{equation}
  We may assume that the elements at the positions $(1, k)$ and
  $(k, 1)$ are nonzero.  Otherwise, we may skip this transformation.

  Since the element at the position $(1, k)$ is transformed only from
  the right-hand side (and the element at the position $(k, 1)$ only
  from the left-hand side), the new elements at these positions are
  \begin{align*}
    m'_{1k} & = m_{1k} \cos \varphi_{k,k+1} + m_{1,k+1} \sin \varphi_{k,k+1}, \\
    m'_{k1} & = -(m_{1k} \cos \varphi_{k,k+1} + m_{1,k+1} \sin \varphi_{k,k+1}) = -m'_{1k}.
  \end{align*}
  By choosing
  \begin{equation}
    \cot \varphi_{k,k+1} = - \frac{m_{1,k+1}}{m_{1k}},
    \label{2.4}
  \end{equation}
  from the basic identity for the trigonometric functions
  $\sin^2 \varphi_{k,k+1} + \cos^2 \varphi_{k,k+1} = 1$, it is easy to
  derive that the sines and the cosines in (\ref{2.3}) (which
  annihilate $m'_{1,k}$) are
  \begin{displaymath}
    \sin \varphi_{k,k+1} = \pm \frac{1}{\sqrt{1 + \cot^2 \varphi_{k,k+1}}}, \qquad
    \cos \varphi_{k,k+1} = \sin \varphi_{k,k+1} \cot \varphi_{k,k+1},
  \end{displaymath}
  where $\cot \varphi_{k,k+1}$ is defined by (\ref{2.4}).

  The next step is to annihilate the elements at the positions
  $(2, k-1)$ and $(k-1, 2)$ by a rotation in the plane $(k-1, k)$.
  This transformation will not destroy the zero pattern, since the
  rows/columns $k-1$ and $k$ already have zeroes as the first elements
  in the corresponding row/column.
  
  In a similar way all the elements of the $k$th antidiagonal will
  be annihilated without destroying the already introduced zeroes.
  
  After the annihilation in this step we obtain $M_k$, which has the
  same form as $M_{k-1}$ from (\ref{2.1}), but the matrix
  $M_{11}^{}$, still antitriangular, has one row and one column more
  than the matrix $M_{11}^{}$ from (\ref{2.2}).  This was the step of
  the induction.

  We proceed with the annihilation of one antidiagonal after another
  until $k$ becomes $n$.
\end{proof}

As one can expect, since the skew-symmetric matrices have the
eigenvalues in pairs of the form $\pm \lambda i$, one `positive' and
one `negative' on the imaginary axis, there is no submatrix $X$ in the
symmetric block antitriangular form (\ref{1.1}), whose dimension
corresponds to the difference between the number of positive and
negative eigenvalues of the symmetric matrix.

If a skew-symmetric matrix $A$ of order $n = 2p$ is given by its
antitriangular factor, then the determinant of $A$ is
\begin{align*}
  \det(A) & = \det(Q M Q^T) = \det(M) \\
  & = (-1)^{2p+1} \cdot (-1)^{2p} \cdots (-1)^3 \cdot (-1)^2 \cdot
  (-1)^p m_{1,2p}^2 m_{2,2p-1}^2 \cdots m_{p,p+1}^2 \\
  & = (-1)^{2(p^2 +2p)} m_{1,2p}^2 m_{2,2p-1}^2 \cdots m_{p,p+1}^2
  = m_{1,2p}^2 m_{2,2p-1}^2 \cdots m_{p,p+1}^2.
\end{align*}
Therefore, $A$ (of even order) is singular if and only if at least one
of the antidiagonal entries is zero.  If $A$ is of odd order, one of
the zeroes of the main diagonal is on the antidiagonal, which proves
the well-known fact that any skew-symmetric matrix of odd order is
always singular.  Now suppose that $A$ is of even order and singular,
and the antidiagonal entry at the position
$(\ell, n - \ell + 1)$, $\ell \leq n - \ell + 1$ is zero.  Obviously,
due to skew-symmetry, the element at the position
$(n - \ell + 1, \ell)$ is also zero.  If there is more than one pair
of zeroes on the antidiagonal, we start from a zero with the smallest
difference of its column and row indices.

Now we apply a procedure similar to the procedure of annihilation of
the elements of the antidiagonal from the previous theorem, but
starting with the annihilation of the element at the position
$(\ell + 1, n - \ell)$ by a rotation in the plane
$(n - \ell, n - \ell + 1)$.  This rotation will also annihilate the
element at the position $(n - \ell, \ell + 1)$.  We proceed with this
annihilation process until all the elements on the antidiagonal
between $(\ell, n - \ell + 1)$ and $(n - \ell + 1, \ell)$ are zeroes.

If $A$ is of even order, after the previous sequence of
transformations, our matrix has a middle part of the antidiagonal
equal to zero.  After such a preparation, a procedure for the
annihilation of the nonzero elements on the antidiagonal is similar
for odd and even orders.  If $A$ is of odd order, the elements at the
positions $(\lceil n / 2 \rceil + 1, \lceil n / 2 \rceil - 1)$ and
$(\lceil n / 2 \rceil - 1, \lceil n / 2 \rceil + 1)$ are the first to
be annihilated, by a rotation in the plane
$(\lceil n / 2 \rceil - 1, \lceil n / 2 \rceil)$.  If $A$ is of even
order, we proceed with the annihilation of the elements at the
positions $(n - \ell + 2, \ell - 1)$ and $(\ell - 1, n - \ell + 2)$ by
a rotation in the plane $(\ell - 1, \ell)$.  The process is finished
when the elements at the positions $(1, n)$ and $(n, 1)$ are
annihilated by a rotation in the plane $(1, 2)$.

If all the elements on the first nontrivial antidiagonal of the final
matrix are nonzero, the matrix has rank $n-1$.  Otherwise, we continue
the process until all elements of some antidiagonal are nonzero.  The
count of such elements is the rank of the matrix.

The process of detecting the rank is illustrated in
Figures~\ref{fig:1} and~\ref{fig:2}.  The first of them is for a
matrix of even order, and the second for a matrix of odd order.

\begin{figure}[hbt]
  \centering{%
  \includegraphics{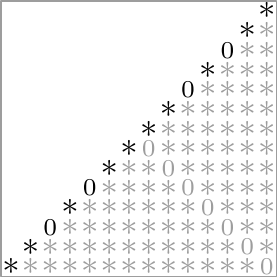} \quad
  \includegraphics{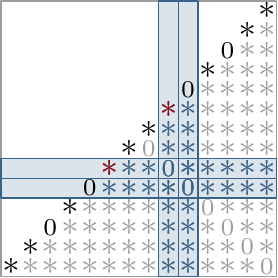} \quad
  \includegraphics{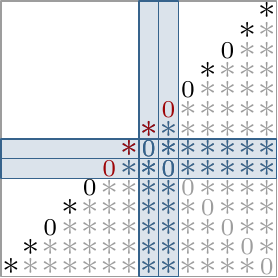} \quad
  \includegraphics{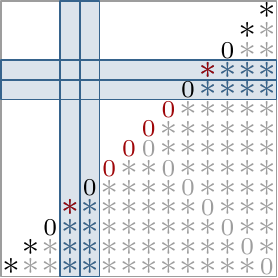} \\[6pt]
  \includegraphics{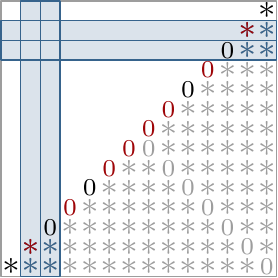} \quad
  \includegraphics{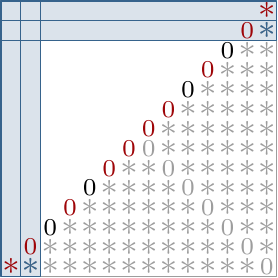} \quad
  \includegraphics{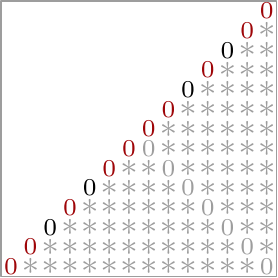}}
  \caption{From left to right, top to bottom -- annihilation of the
    antidiagonal of a matrix of \textbf{even} order: the first
    subfigure is the state before annihilation, while the last is
    after the completion of the process for the first antidiagonal.
    The horizontal and the vertical stripes show the application of
    the Givens rotations from the left and from the right,
    respectively.}
  \label{fig:1}
\end{figure}

\begin{figure}[hbt]
  \centering{%
  \includegraphics{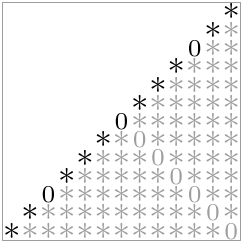} \quad
  \includegraphics{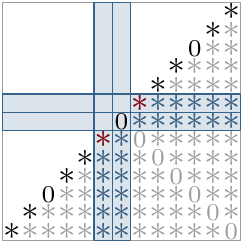} \quad
  \includegraphics{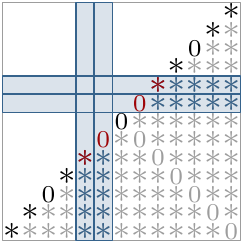} \quad
  \includegraphics{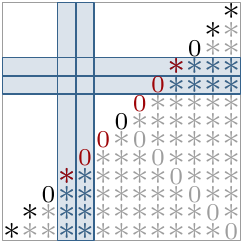} \\[6pt]
  \includegraphics{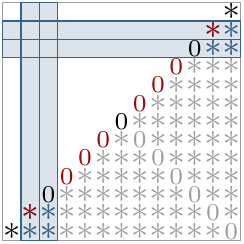} \quad
  \includegraphics{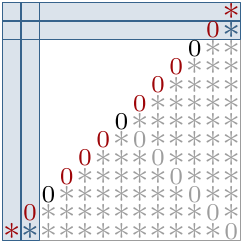} \quad
  \includegraphics{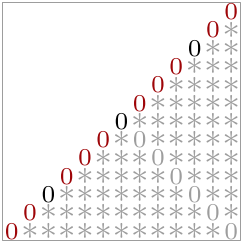}}
  \caption{From left to right, top to bottom -- annihilation of the
    antidiagonal of a matrix of \textbf{odd} order: the first
    subfigure is the state before annihilation, while the last is
    after the completion of the process for the first antidiagonal.
    The horizontal and the vertical stripes show the application of
    the Givens rotations from the left and from the right,
    respectively.}
  \label{fig:2}
\end{figure}
%
%%%%%%%%%%%%%%%%%%%%%%%%%%%%%%%%%%%%%%%%%%%%%%%%%%%%%%%%%%%%%%%%%%%%%%%%%%%%%%
%
\section{Numerical computation of the antitriangular form of a skew-symmetric matrix}
\label{sec:3}
%
%%%%%%%%%%%%%%%%%%%%%%%%%%%%%%%%%%%%%%%%%%%%%%%%%%%%%%%%%%%%%%%%%%%%%%%%%%%%%%
%
When the QR factorization is used for the numerical rank detection, it
is always computed with column pivoting.  Here we derive a similar
algorithm for the antitriangular factorization.  For a purely
practical reason we reduce an antitriangular matrix to the upper
antitriangular form, which can easily be `flipped' over the main
antidiagonal to the lower antitriangular form.

In addition to the procedure described in the previous section, here
we derive a reduction to the antitriangular form by applying the
ordinary Householder reflectors.

Before the annihilation process in each step, a pivot column is
chosen.  The pivot column has maximal norm in the unreduced part of
the matrix.  In the first step, the unreduced part is the whole
matrix.  Then the whole pivot column (not only its unreduced part) is
swapped with the last column in the whole matrix by a permutation
$P_1^T$, applied from the right, while $P_1$ is applied from the left
to swap the corresponding rows of the matrix.

An orthogonal matrix $H_1$, that consists of a Householder reflector
$\widetilde{H}_1$ of order $n-1$ complemented with the identity
matrix of order $1$,
\begin{displaymath}
  H_1 = \diag(\widetilde{H}_1, 1),
\end{displaymath}
is then applied to the first $n-1$ rows of $A$ such that the last
column is reduced to a single element at the position $(1, n)$.  Note
that this element is the largest by absolute value in the matrix
$H_1^{} P_1^{} A P_1^T$.  After the completion of the left-hand-side
transformation, the right-hand side transformation with the same $H_1$
(since $H_1^T = H_1^{}$) is applied from to the first $n-1$ columns of
$H_1^{} P_1^{} A P_1^T$.

In the second step we proceed by reducing the last-but-one row and
column of $\widehat{H}_1^{} A \widehat{H}_1^T$, where
$\widehat{H}_1 = H_1 P_1$, while the first and the last rows and
columns of the whole matrix remain intact.  After the appropriate
pivoting by a permutation $P_2$, an orthogonal matrix
$H_2$,
\begin{displaymath}
  H_2 = \diag(1, \widetilde{H}_2, 1),
\end{displaymath}
where $\widetilde{H}_2$ is a Householder reflector, is chosen such
that the submatrix
$(\widehat{H}_2^{} \widehat{H}_1^{} A \widehat{H}_1^T \widehat{H}_2^T)(2:n-1,2:n-1)$,
with $\widehat{H}_2 = H_2 P_2$, has its last column (and row) equal to
$c e_1$ ($-c e_1^T$), where $|c|$ is the norm of the unreduced part of
the pivot (now, the penultimate) column.

The process is repeated in the same way until the unreduced part of
the pivot column is of length $1$, as shown in Algorithm~\ref{alg:1},
while the first steps of the reduction process are illustrated in
Figure~\ref{fig:3}.

\begin{algorithm}[hbt]
\SetKwInput{KwIn}{Input}
\SetKwInput{KwOut}{Output}
\KwIn{$A$, a skew-symmetric matrix of order $n$, and $\tol$, a numerical tolerance.}
\KwOut{$A$ reduced to the upper antitriangular form.}
\BlankLine
\Begin{
    \For{$\step = 1$ \KwTo $n/2$}
        {$i1 = \step$; \quad $i2 = n - \step + 1$\;
          \For{$k = i1$ \KwTo $i2$}
              {compute $n_k = \| A(i1:i2,k) \|_2$
              }
          compute $\imax$ -- the index of the column with the largest norm, $n_{\imax} = \max_{k = i1, \ldots, i2} n_k$\;
          \If{$\imax \neq i2$}
             {swap $A(:,\imax)$ and $A(:,i2)$\;
              swap $A(\imax,:)$ and $A(i2,:)$\;
             }
          \If{$n_k > \tol$}
             {compute the Householder reflector $\widetilde{H}_{\step}^{}$ from $A(i1:i2,i2)$\;
              apply $\widetilde{H}_{\step}^{}$ from the left to the columns $1$ to $i2$\;
              set $A(i1+1:i2,i2) = 0$\;
              apply $\widetilde{H}_{\step}^{}$ from the right to the rows $1$ to $i2$\;
              set $A(i2,i1+1:i2-1) = 0$\;
              skew-symmetrize matrix $A(i1:i2,i1:i2) = (A(i1:i2,i1:i2) - A^T(i1:i2,i1:i2))/2$
             }
        }
  }
\caption{Reduction to the upper antitriangular form.}
\label{alg:1}
\end{algorithm}

\begin{figure}[hbt]
  \centering{%
  \includegraphics{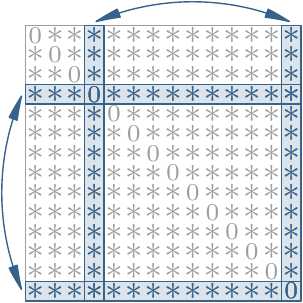} \quad
  \includegraphics{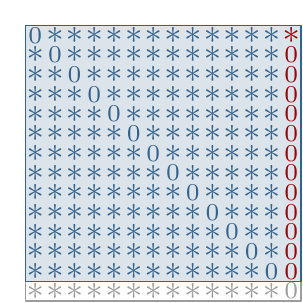} \quad
  \includegraphics{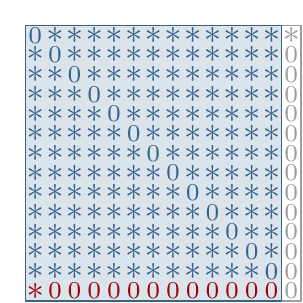} \\[6pt]
  \includegraphics{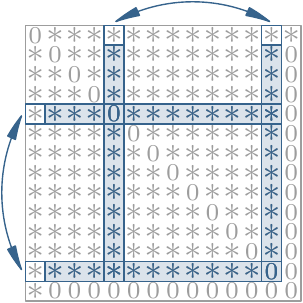} \quad
  \includegraphics{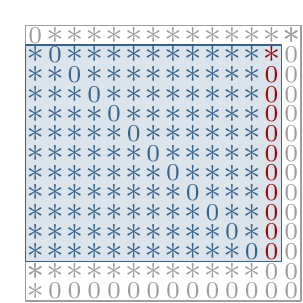} \quad
  \includegraphics{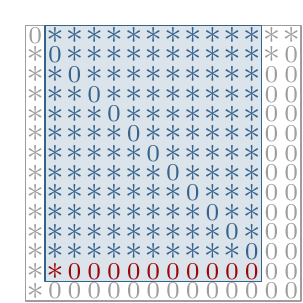} \\[6pt]
  \includegraphics{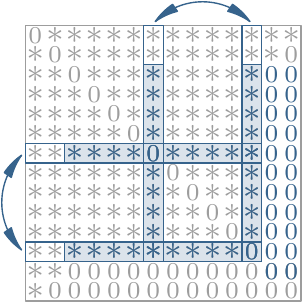} \quad
  \includegraphics{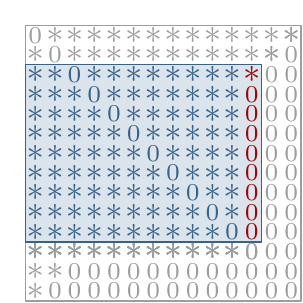} \quad
  \includegraphics{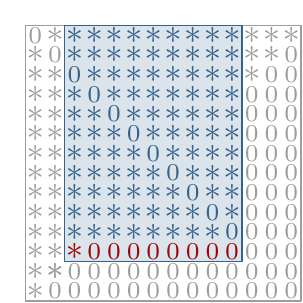}}
  \caption{The first three steps of the reduction process.  The first
    subfigure in each row shows the symmetric pivoting.  The second
    shows the application of the left-hand-side, while the third shows
    the right-hand-side orthogonal transformation.  The shaded regions in
    the second and the third subfigure show the part of the matrix
    affected by the Householder reflector $\widetilde{H}_{\step}$.}
  \label{fig:3}
\end{figure}

Note that Algorithm~\ref{alg:1} could be written to work only on one
triangle of the matrix, as is customary in LAPACK.  In that case the
skew-symmetrization in the last step of the algorithm would not be
needed.  Also, the generating vector $v_j$ of the Householder
reflector $\widetilde{H}_j^{} = I - v_j^{} v_j^T$ could be stored
after the $j$th step below the main diagonal in the upper
antitriangular case, i.e., the $k$th element of $v_j$ in the place
$A(j+k,j)$.

Algorithm~\ref{alg:1} can be stopped earlier if the pivot column norm
in the unreduced part of the matrix is (numerically) zero.  Since this
is the largest column norm in the unreduced part, the whole submatrix
is then zero.  Therefore, the right-hand-side transformation will not
spoil the zeroes in this submatrix.

Suppose that the reduction process illustrated in Figure~\ref{fig:3}
is completed, i.e., the norms of the unreduced part of the columns are
zeroes.  This situation is displayed in Figure~\ref{fig:4}.

\begin{figure}[hbt]
  \centering{%
  \includegraphics{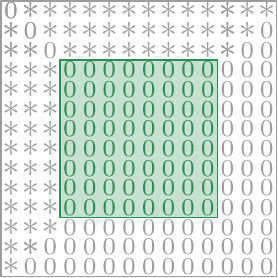}}
  \caption{A possible termination of the antitriangularization
    process.}
  \label{fig:4}
\end{figure}

Of course, the situation in the real process is not as ideal as in the
Figure~\ref{fig:4} since the small rounding errors shift the zeroes in
the shaded region to the elements with the small absolute values, such
that the column norm of each shaded column is less than or equal to
some $\tol$.  The first question is how to choose $\tol$.  An
experience from similar factorizations shows that $\tol$ should
include the machine epsilon, $\varepsilon$, the number of
transformations applied to each element (of order $n$), and the
largest element in the process.  Since the norms in the first step are
chosen such that the norm of the last column is the largest, this is a
good candidate for the largest element in the process.  Therefore,
$\tol$ is set as
\begin{displaymath}
  \tol = n \varepsilon \max_{k=1,\ldots,n} \| A(:,k) \|_2.
\end{displaymath}
Note that the determination of $\tol$ is directly related to the
determination of the rank of the matrix.  After the selection of
$\tol$, the elements in the shaded region in Figure~\ref{fig:4} should
be set to zeroes.

As we have already seen in the previous section, this antitriangular
form can be further reduced to an antitriangular form with nonzero
antidiagonal elements.  Once again, this process can be done using
Householder reflectors.

First, denote by $\ell$, $\ell \leq n/2$, the last column with a
nonzero antidiagonal element.  Then apply an orthogonal transformation
\begin{displaymath}
  H'_\ell = \diag(I_\ell, \widetilde{H}'_\ell, I_{\ell-1}),
\end{displaymath}
where $\widetilde{H}'_\ell$ is the Householder reflector of order
$m \assgn n-2\ell+1$, from the left to reduce the part of the $\ell$th
column (from the $(\ell+1)$th row) to a single element at the position
$(\ell+1,\ell)$.  Due to skew-symmetry, the application of the
right-hand-side transformation is not needed.  The elements of the
first $\ell$ columns that have been transformed could be transposed,
with a change of sign, and written to the elements of the first $\ell$
rows.

The next transformation $H'_{\ell-1}$,
\begin{displaymath}
  H'_{\ell-1} = \diag(I_{\ell+1}, \widetilde{H}'_{\ell-1}, I_{\ell-2}),
\end{displaymath}
has its Householder reflector $\widetilde{H}'_{\ell-1}$ of the same
order $m$ as $\widetilde{H}'_\ell$, but ``shifted down'' one place.
This transformation is applied to the part of $(\ell-1)$th column
(from row $\ell+2$) to reduce it to a single element at the position
$(\ell+2, \ell-1)$.

This sequence of transformations ends after transforming the first
column (row).  Algorithm~\ref{alg:2} describes this reduction. An
illustration of the reduction process for the matrix from
Figure~\ref{fig:4} is given in Figure~\ref{fig:5}.

\begin{algorithm}
\SetKwInput{KwIn}{Input}
\SetKwInput{KwOut}{Output}
\SetKw{KwBreak}{break}
\KwIn{$A$, a skew-symmetric matrix of order $n$.}
\KwOut{$A$ reduced to the upper antitriangular form with a nontrivial antidiagonal.}
\BlankLine
\Begin{
    $\tol = n \cdot \varepsilon \cdot \max_{k=1,\ldots,n} \|A(:,k)\|_2$\;
    \For{$\step = 1$ \KwTo $n/2$}
        {$i1 = \step$; \quad $i2 = n - \step + 1$\;
          \For{$k = i1$ \KwTo $i2$}
              {compute $n_k = \| A(i1:i2,k) \|_2$
              }
          compute $\imax$ -- the index of the column with the largest norm, $n_{\imax} = \max_{k = i1, \ldots, i2} n_k$\;
          \If{$\imax \neq i2$}
             {swap $A(:,\imax)$ and $A(:,i2)$\;
              swap $A(\imax,:)$ and $A(i2,:)$\;
             }
          \eIf{$n_k > \tol$}
             {compute the Householder reflector $\widetilde{H}_{\step}^{}$ from $A(i1:i2,i2)$\;
              apply $\widetilde{H}_{\step}^{}$ from the left to the columns $1$ to $i2$\;
              set $A(i1+1:i2,i2) = 0$\;
              apply $\widetilde{H}_{\step}^{}$ from the right to the rows $1$ to $i2$\;
              set $A(i2,i1+1:i2-1) = 0$\;
              skew-symmetrize matrix $A(i1:i2,i1:i2) = (A(i1:i2,i1:i2) - A^T(i1:i2,i1:i2))/2$
             }
             {$A(i1:i2,i1:i2)=0$\;
              \KwBreak\;
             }
        }
    $p1 = \step - 1$; \quad $p2 = n - \step$\;
    \For{$\ell = \step - 1$ \KwTo $1$}
        {$p1 = p1 + 1$; \quad $p2 = p2 + 1$\;
         compute the Householder reflector $\widetilde{H}'_{\ell}$ from $A(p1:p2,\ell)$\;
         apply $\widetilde{H}'_{\ell}$ from the left to the columns $1$ to $\ell$\;
         set $A(p1+1:p2,\ell) = 0$\;
         $A(1:\ell,p1:p2) = -A^T(p1:p2,1:\ell)$\;
        }
  }
\caption{Reduction to the upper antitriangular form with a nonzero diagonal.}
\label{alg:2}
\end{algorithm}

\begin{figure}[hbt]
  \centering{%
  \includegraphics{fig24.pdf} \quad
  \includegraphics{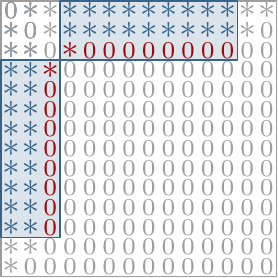} \quad
  \includegraphics{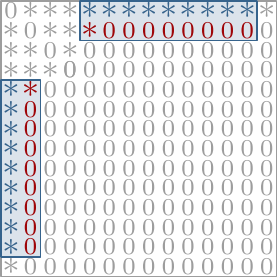} \quad
  \includegraphics{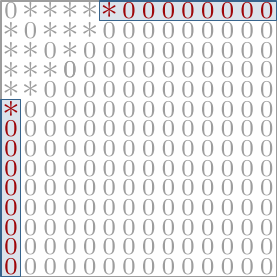} \quad
  \includegraphics{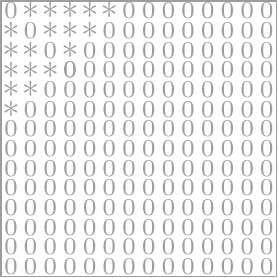}}
  \caption{The first subfigure shows the initial state of a matrix
    before the first iteration of the $\ell$-loop of
    Algorithm~\ref{alg:2}.  The next three subfigures show the effects
    of the transformations $H_3'$, $H_2'$, and $H_1'$, in that order.
    The last subfigure shows the fully reduced matrix.}
  \label{fig:5}
\end{figure}

Algorithm~\ref{alg:2} was tested for the matrices with various ranks
in Fortran's \texttt{double precision}.  The test matrices were
constructed by a procedure similar to the LAPACK's \texttt{dlarge} by
setting their eigenvalues in the Murnaghan form,
\begin{displaymath}
  D = \diag\begin{pmatrix}
  \begin{bmatrix}
    0 & \lambda_1 \\
    -\lambda_1 & 0
  \end{bmatrix},
  \begin{bmatrix}
    0 & \lambda_2 \\
    -\lambda_2 & 0
  \end{bmatrix},
  \ldots,
  \begin{bmatrix}
    0 & \lambda_r \\
    -\lambda_r & 0
  \end{bmatrix},
  0, \ldots, 0
  \end{pmatrix}.
\end{displaymath}
To avoid the unnecessary errors, the random orthogonal matrices
$Q_1, \ldots, Q_n$ were generated in quadruple precision and then
applied to $D$,
\begin{displaymath}
  A' = Q_n^{} \cdots Q_1^{} D Q_1^T \cdots Q_n^T.
\end{displaymath}
Final $A$ was obtained from $A'$ by rounding the quadruple precision
result to \texttt{double precision}.

Our test collection consists of matrices $A_k$ of order $108$ with
ranks of $2, 4, 6, \ldots, 108$.  $A_k$ of rank $2r$ has nonzero
eigenvalues $\{ \pm i, \pm 2^{-1} i, \ldots, \pm 2^{r-1} i \}$.  Note
that for the matrices of higher ranks, eigenvalues are gradually
tending to zero, and therefore are very hard to detect.  The following
results were obtained.
\begin{center}
  \begin{tabular}[hbt]{@{}cccc@{}}
    \toprule
    Matrix rank   & 2--96 & 98 & 100--108\\
    \midrule
    Detected rank & correct & 96 & 98\\
    \bottomrule
  \end{tabular}
\end{center}

The test collection and a Fortran implementation of the
Algorithms~\ref{alg:1} and \ref{alg:2} is freely available at
\texttt{https://github.com/venovako/ATFact} repository.
%
%%%%%%%%%%%%%%%%%%%%%%%%%%%%%%%%%%%%%%%%%%%%%%%%%%%%%%%%%%%%%%%%%%%%%%%%%%%%%%
%
\section{Multi-arrowhead form of a skew-symmetric matrix}
\label{sec:4}
%
%%%%%%%%%%%%%%%%%%%%%%%%%%%%%%%%%%%%%%%%%%%%%%%%%%%%%%%%%%%%%%%%%%%%%%%%%%%%%%
%
In Section~\ref{sec:2} we transform a full skew-symmetric matrix to
antitriangular form.  From the antitriangular form of a skew-symmetric
matrix it is easy to obtain a new form -- the multi-arrowhead form of
a matrix.
\begin{theorem}
  Let $M \in \mathbb{R}^{n \times n}$ be a skew-symmetric matrix in
  the antitriangular form.  By the two-sided permutations $P$,
  \begin{equation}
    P = \begin{cases}
      [e_k, e_{k-1}, e_{k+1}, e_{k-2}, e_{k+2} , \ldots, e_1, e_n], &
      \text{if $n = 2k - 1$,} \\
      [e_k, e_{k+1}, e_{k-1}, e_{k+2}, e_{k-2} , \ldots, e_1, e_n], &
      \text{if $n = 2k$,}
    \end{cases}
    \label{4.1}
  \end{equation}
  the matrix $M$ can be transformed into
  \begin{displaymath}
    M = P S P^T,
  \end{displaymath}
  where $S$ has the following multi-arrowhead form.  If $n$ is odd, then
  \begin{displaymath}
    S = \begin{bmatrix}
      0     & 0      & s_{13}  & 0      & s_{15}  & 0 & \cdots & 0 & s_{1n} \\
      0     & 0      & s_{23}  & 0      & s_{25}  & 0 & \cdots & 0 & s_{2n} \\
     -s_{13} & -s_{23} & 0      & 0      & s_{35}  & 0 & \cdots & 0 & s_{3n} \\
      0     & 0      & 0      & 0      & s_{45}  & 0 & \cdots & 0 & s_{4n} \\
     -s_{15} & -s_{25} & -s_{35} & -s_{45} & 0      & 0 & \cdots & 0 & s_{5n} \\
      0     & 0      & 0      & 0      & 0      & 0 & \cdots & 0 & s_{6n} \\
     \vdots & \vdots & \vdots & \vdots & \vdots & \vdots & \ddots & \vdots & \vdots \\
      0     & 0      & 0      & 0      & 0      & 0 & \cdots & 0 & s_{n-1,n} \\
     -s_{1n} & -s_{2n} & -s_{3n} & -s_{4n} & -s_{5n} & -s_{6n} & \cdots & -s_{n-1,n} & 0
    \end{bmatrix},
  \end{displaymath}
  and if $n$ is even, then
  \begin{displaymath}
    S = \begin{bmatrix}
      0     & s_{12}  & 0      & s_{14}  & 0 & \cdots & 0 & s_{1n} \\
     -s_{12} & 0      & 0      & s_{24}  & 0 & \cdots & 0 & s_{2n} \\
      0     & 0      & 0      & s_{34}  & 0 & \cdots & 0 & s_{3n} \\
     -s_{14} & -s_{24} & -s_{34} & 0      & 0 & \cdots & 0 & s_{4n} \\
      0     & 0      & 0      & 0      & 0 & \cdots & 0 & s_{5n} \\
     \vdots & \vdots & \vdots & \vdots & \vdots & \ddots & \vdots & \vdots \\
      0     & 0      & 0      & 0      & 0 & \cdots & 0 & s_{n-1,n} \\
     -s_{1n} & -s_{2n} & -s_{3n} & -s_{4n} & -s_{5n} & \cdots & -s_{n-1,n} & 0
    \end{bmatrix}.
  \end{displaymath}
  Moreover, if $n$ is odd, the first row and the first column can
  become a zero row and a zero column with an additional sequence of
  rotations at the positions $(1,2)$, $(1,4)$, \ldots, $(1, n-1)$.
  \label{thm4.1}
\end{theorem}

\begin{proof}
  The required result is obtained by a symmetric permutation
  $P^T M P$, where $P$ is given by~(\ref{4.1}).

  The remaining part of the proof for the skew-symmetric matrices of
  odd order is straightforward.  By a rotation at the position
  $(1, 2)$ we annihilate the elements at the positions $(1, 3)$ and
  $(3, 1)$.  Then we use a rotation at the position $(1, 4)$ and
  annihilate the elements at the positions $(1, 5)$ and $(5, 1)$, and
  so on until the rotation at the position $(1, n-1)$ which
  annihilates the elements at the positions $(1, n)$ and $(n, 1)$.
\end{proof}
%
%%%%%%%%%%%%%%%%%%%%%%%%%%%%%%%%%%%%%%%%%%%%%%%%%%%%%%%%%%%%%%%%%%%%%%%%%%%%%%
%
\section{Factorization of a skew-Hermitian matrix into the block antitriangular form}
\label{sec:5}
%
%%%%%%%%%%%%%%%%%%%%%%%%%%%%%%%%%%%%%%%%%%%%%%%%%%%%%%%%%%%%%%%%%%%%%%%%%%%%%%
%
Skew-Hermitian matrices are the complex generalizations of the
skew-symmetric matrices, with purely imaginary eigenvalues, but now
they need not be in complex-conjugate pairs.  Therefore, we can have a
surplus of `positive' or `negative' signs on the imaginary axis.

For example, if $Q$ is any unitary matrix, then a matrix $A = ai I$,
where $a \in \mathbb{R}$ and $a \neq 0$, cannot be transformed into
antitriangular form since $Q^{\ast} A Q = ai I$.

On the other hand, $H \assgn iA$ is a Hermitian matrix if $A$ is
skew-Hermitian.  Therefore, if $H$ can be transformed into block
antitriangular form, a relation between skew-Hermitian and Hermitian
matrices is used to obtain the the block antitriangular form of $A$.

If we look at the proof of Theorem~2.1
from~\cite{Mastronardi-VanDooren-13}, that theorem is also valid for
the Hermitian matrices if in the statement of the Theorem orthogonal
matrices are replaced by unitary matrices and the transpose operation
is replaced by the conjugate transpose.  That proof relies on the
properties of the nonnegative, nonpositive, neutral and null-spaces.
In~\cite{Gohberg-Lancaster-Rodman-05}, all the required properties are
derived, not only for the complex Euclidean scalar products, but for
the indefinite complex scalar products.  Therefore, it is easy to
prove the following theorem.
\begin{theorem}
  Let $A \in \mathbb{C}^{n \times n}$ be a Hermitian indefinite matrix
  with $\inertia(A) = (n_{-}, n_0, n_{+})$,
  $n_1 = \min(n_{-}, n_{+})$, $n_2 = \max(n_{-}, n_{+}) - n_1$.
  Then, there exists a unitary matrix $Q \in \mathbb{C}^{n \times n}$
  such that
  \begin{displaymath}
    M = Q^{\ast} A Q = \begin{bmatrix}
      0 & 0 & 0 & 0 \\
      0 & 0 & 0 & Y^{\ast} \\
      0 & 0 & X & Z^{\ast} \\
      0 & Y & Z & W
    \end{bmatrix},
  \end{displaymath}
  where $Y \in \mathbb{C}^{n_1 \times n_1}$ is nonsingular and lower
  antitriangular, $W \in \mathbb{C}^{n_1 \times n_1}$ is Hermitian,
  $X \in \mathbb{C}^{n_2 \times n_2}$ is Hermitian and definite, and
  $Z \in \mathbb{C}^{n_1 \times n_2}$.
  \label{thm5.1}
\end{theorem}

In the next Corollary we abuse the notation for the inertia of the
skew-Hermitian matrices.  If the skew-Hermitian matrix $A$ has $n_{-}$
eigenvalues on the negative part of the imaginary axis, $n_0$ zeroes
as eigenvalues and $n_{+}$ eigenvalues on the positive part of the
imaginary axis, we denote this by
$\inertia(A) = i (n_{-}, n_0, n_{+})$.
\begin{corollary}
  Let $A \in \mathbb{C}^{n \times n}$ be a skew-Hermitian matrix, and
  let $\inertia(A) = i (n_{-}, n_0, n_{+})$, such that neither
  $n_{-} = n$, nor $n_{+} = n$, and $n_1 = \min(n_{-}, n_{+})$,
  $n_2 = \max(n_{-}, n_{+}) - n_1$.  Then, there exists a unitary
  matrix $Q \in \mathbb{C}^{n \times n}$ such that
  \begin{equation}
    M = Q^{\ast} A Q = \begin{bmatrix}
      0 & 0 & 0 & \hphantom{-}0 \\
      0 & 0 & 0 & -Y^{\ast} \\
      0 & 0 & X & -Z^{\ast} \\
      0 & Y & Z & \hphantom{-}W
    \end{bmatrix},
    \label{5.1}
  \end{equation}
  where $Y \in \mathbb{C}^{n_1 \times n_1}$ is nonsingular and lower
  antitriangular, $W \in \mathbb{C}^{n_1 \times n_1}$ and
  $X \in \mathbb{C}^{n_2 \times n_2}$ are skew-Hermitian, and
  $Z \in \mathbb{C}^{n_1 \times n_2}$.  Then, either
  $\inertia(X) = i (n_2, 0, 0)$ or $\inertia(X) = i (0, 0, n_2)$.
  \label{cor5.2}
\end{corollary}

\begin{proof}
  If the previous Theorem~\ref{thm5.1} is applied to $H = iA$, it holds
  \begin{equation}
    \widetilde{M} = Q^{\ast} H Q = \begin{bmatrix}
      0 & 0 & 0 & 0 \\
      0 & 0 & 0 & \widetilde{Y}^{\ast} \\
      0 & 0 & \widetilde{X} & \widetilde{Z}^{\ast} \\
      0 & \widetilde{Y} & \widetilde{Z} & \widetilde{W}
    \end{bmatrix}.
    \label{5.2}
  \end{equation}
  If (\ref{5.2}) is multiplied by $-i$ we obtain (\ref{5.1}) by
  setting $M = -i \widetilde{M}$, $Y \assgn -i\widetilde{Y}$,
  $X \assgn -i\widetilde{X}$, $Z \assgn -i\widetilde{Z}$,
  $W \assgn -i\widetilde{W}$.  It is easy to see that $X = -X^{\ast}$
  and $W = -W^{\ast}$.  According to Theorem~\ref{thm5.1}, the matrix
  $\widetilde{X}$ is definite, therefore all eigenvalues of $X$,
  \begin{displaymath}
    \lambda_k(X) = \lambda_k(-i \widetilde{X}) = -i \lambda_k(\widetilde{X}),
  \end{displaymath}
  are placed at the same part of the imaginary axis.
\end{proof}
%
%%%%%%%%%%%%%%%%%%%%%%%%%%%%%%%%%%%%%%%%%%%%%%%%%%%%%%%%%%%%%%%%%%%%%%%%%%%%%%
%
\section*{Acknowledgments}
\label{sec:A}
%
%%%%%%%%%%%%%%%%%%%%%%%%%%%%%%%%%%%%%%%%%%%%%%%%%%%%%%%%%%%%%%%%%%%%%%%%%%%%%%
%
We wish to express our gratitude for an insightful anonymous review,
which motivated us to develop the effective computational procedure
described in Section~\ref{sec:3}.  Also, we are grateful to Vedran
Novakovi\'{c} for his help with developing the algorithms and
improving the presentation of this text.
%
%%%%%%%%%%%%%%%%%%%%%%%%%%%%%%%%%%%%%%%%%%%%%%%%%%%%%%%%%%%%%%%%%%%%%%%%%%%%%%
%
%% \bibliographystyle{elsart-num-sort}
%% \bibliography{ref}

\end{document}